\documentclass[12pt]{amsart}

\usepackage{latexsym,amssymb,amsmath,amsthm,amscd,graphicx, xcolor, enumerate}

\usepackage{mathrsfs}
\usepackage{times}
\usepackage{microtype}
\usepackage{setspace}
\usepackage[margin=1.2in]{geometry}

\usepackage{cite}

\usepackage[colorlinks=true, pdfstartview=FitV, linkcolor=blue,
citecolor=blue, urlcolor=blue]{hyperref}

\numberwithin{equation}{section}
\allowdisplaybreaks

\newtheorem{thm}{Theorem}[section]
\newtheorem{lemma}[thm]{Lemma}
\newtheorem{corollary}[thm]{Corollary}

\theoremstyle{definition}

\newtheorem{remark}[thm]{Remark}
\newtheorem{definition}[thm]{Definition}
\newtheorem{example}[thm]{Example}

\theoremstyle{remark}

\def\XXint#1#2#3{{\setbox0=\hbox{$#1{#2#3}{\int}$}
     \vcenter{\hbox{$#2#3$}}\kern-.5\wd0}}

\begin{document}

\title[DENSITY OF NEURAL NETWORK CLASSES ON COMPACT SUBSETS OF TVS]{DENSITY OF NEURAL NETWORK CLASSES ON COMPACT SUBSETS OF TOPOLOGICAL VECTOR SPACES}

\author[Mohammad javad Baghbanbashi]{Mohammad javad Baghbanbashi}

\address{  Mohammad javad Baghbanbashi\\
Department of Mathematics\\
    Institute for Advanced Studies in Basic Sciences (IASBS) \\
    Zanjan, 45137-66731 \\
    Iran}
	\email{baghban.mj@iasbs.ac.ir}

\author[A. Ghorbanalizadeh]{Arash Ghorbanalizadeh}

\address{Arash Ghorbanalizadeh\\
	    Department of Mathematics\\
    Institute for Advanced Studies in Basic Sciences (IASBS) \\
    Zanjan, 45137-66731 \\
    Iran}
   
    \email{ gurbanalizade@gmail.com; ghorbanalizadeh@iasbs.ac.ir}	

\email{}

\begin{abstract}
We prove density results for neural-network classes on compact sets \(K\subset X\),
where \(X\) is a topological vector space whose continuous dual \(X^*\) separates
points. Let \(\Psi:\mathbb R\to\mathbb R\) be a continuous squashing function. We show
that the class
\[
\Sigma_X(\Psi)
=
\left\{
\sum_{j=1}^{N}\omega_j\Psi(f_j(x)+b_j):
N\in\mathbb N,\ \omega_j,b_j\in\mathbb R,\ f_j\in X^*
\right\}
\]
is dense in \(C(K)\) with respect to the uniform norm. As a consequence, if \(\mu\) is a
Radon probability measure supported on \(K\), then \(\Sigma_X(\Psi)\) is dense in
\(L^p(K,\mu)\) for every \(1\le p<\infty\). 
\end{abstract}


\maketitle
\vspace{-0.6cm}

\section{Introduction}

The universal approximation theorem is one of the foundational results in the mathematical
theory of neural networks. In its classical form, it asserts that finite linear combinations
of ridge functions of the form
\[
x\mapsto \Psi(a\cdot x+b), \qquad a\in\mathbb R^n,\ b\in\mathbb R,
\]
are dense in suitable spaces of functions on compact subsets of $\mathbb R^n$, provided
the activation function $\Psi$ satisfies appropriate non-degeneracy assumptions. The
classical works of Cybenko \cite{Cybenko1989} and Hornik, Stinchcombe, and White
\cite{Hornik1989} established this phenomenon by using arguments based on the
Hahn--Banach theorem and the Riesz representation theorem. Further developments and
refinements were obtained by Hornik \cite{Hornik1991} and by Leshno--Lin--Pinkus--
Schocken \cite{Leshno1993}. The survey of Pinkus \cite{Pinkus1999} gives a
comprehensive approximation-theoretic account of these developments.

More recently, approximation by neural networks in variable Lebesgue spaces was studied
in \cite{CapelOcariz2020}, while related density results were revisited from the viewpoint
of Fourier analysis and approximation theory in \cite{IzukiSawano2025}. In the Euclidean
setting, the finite-dimensional structure of $\mathbb R^n$ plays an essential role.
Tools such as coordinates, polynomials, Fourier transforms, and approximate identities
are available, and they are often used to prove that a finite signed measure is zero
whenever it integrates all functions in the relevant neural-network class to zero.

The aim of this paper is to establish an analogue of such density results on compact
sets $K\subset X$, where $X$ is a topological vector space whose continuous dual $X^*$
separates points. Since Euclidean tools such as Fourier transforms and convolution are
not available in this general setting, we use a different method based on finite-dimensional
pushforward measures and the Stone--Weierstrass theorem. In this approach, the role
played by Euclidean coordinates and polynomials is replaced by the algebra generated by
the restrictions of elements of $X^*$ to $K$.

Universal approximation on topological vector spaces was recently considered by
Ismailov \cite{Ismailov2024}, who studied feedforward neural networks with inputs from
a topological vector space. A further related direction is the extension of such
approximation results to networks with outputs in locally convex spaces
\cite{Saini2026}. Our approach is different from Ismailov's: we give a direct proof based
on finite-dimensional pushforward measures associated with continuous linear functionals
and the Stone--Weierstrass theorem. This proof also yields, in a transparent way, density
in \(L^p(K,\mu)\) for Radon probability measures supported on compact subsets.

For a continuous squashing function $\Psi$, we consider the class
\[
\Sigma_X(\Psi)
=
\left\{
\sum_{j=1}^{N}\omega_j\Psi(f_j(x)+b_j):
N\in\mathbb N,\ \omega_j,b_j\in\mathbb R,\ f_j\in X^*
\right\}.
\]
We prove that $\Sigma_X(\Psi)$ is dense in $C(K)$ and, as a consequence, dense in
$L^p(K,\mu)$ for every $1\le p<\infty$ and every Radon probability measure $\mu$
supported on $K$.

The paper is organized as follows. In Section~2, we collect the preliminary tools used
throughout the paper. Section~3 is devoted to the main density results in $C(K)$ and
$L^p(K,\mu)$. In Section~4, we discuss examples and remarks, including the Euclidean,
Banach-space, and locally convex cases.

\section{Preliminaries}
In this section we recall the basic tools used in the sequel. We first state the representation and approximation results needed for the duality argument, and then introduce the neural-network class and the algebra generated by continuous linear functionals.

\begin{thm}[Riesz--Markov--Kakutani Representation Theorem {\cite[Theorem~7.17]{Folland1999}}]\label{thm:rmk}
Let $K$ be a compact Hausdorff space, and let $C(K)$ denote the Banach space of
real-valued continuous functions on $K$ equipped with the uniform norm. Then, for every
continuous linear functional $I\in C(K)^*$, there exists a unique finite signed regular
Borel measure $\mu$ on $K$ such that
\[
I(f)=\int_K f(x)\,d\mu(x), \qquad f\in C(K).
\]
Moreover,
\[
\|I\|_{C(K)^*}=|\mu|(K),
\]
where $|\mu|$ denotes the total variation measure of $\mu$.
\end{thm}

\begin{thm}[Lusin's Theorem {\cite[Section~7.1]{Bogachev2007}}]\label{thm:lusin}
Let $(X,\mathcal B,\mu)$ be a Radon measure space, and let $f:X\to\mathbb R$ be a
measurable function. Then, for every $\varepsilon>0$, there exists a compact set
$F\subset X$ such that
\[
\mu(X\setminus F)<\varepsilon
\]
and $f|_F$ is continuous.
\end{thm}

\begin{thm}[Tietze Extension Theorem {\cite[Section~35]{Munkres2000}}]\label{thm:tietze}
Let $X$ be a normal topological space, let $A\subset X$ be closed, and let
$f:A\to\mathbb R$ be continuous. Then there exists a continuous function
$g:X\to\mathbb R$ such that $g=f$ on $A$ and
\[
\sup_{x\in X}|g(x)|=\sup_{x\in A}|f(x)|.
\]
\end{thm}

\begin{thm}[Stone--Weierstrass Theorem {\cite[Section~7.3]{Rudin1976}}]\label{thm:stone-weierstrass}
Let $K$ be a compact Hausdorff space and let $\mathcal A\subset C(K)$ be a subalgebra
containing the constant functions. Assume that $\mathcal A$ separates points of $K$, that
is, for every distinct $x,y\in K$ there exists $f\in\mathcal A$ such that $f(x)\ne f(y)$.
Then $\mathcal A$ is dense in $C(K)$ with respect to the uniform norm.
\end{thm}

\begin{definition}[Squashing function {\cite{Cybenko1989,Hornik1989}}]\label{def:squashing}
A function $\Psi:\mathbb R\to[0,1]$ is called a squashing function if $\Psi$ is
non-decreasing and
\[
\lim_{t\to\infty}\Psi(t)=1,\qquad \lim_{t\to-\infty}\Psi(t)=0.
\]
\end{definition}

\begin{definition}[Radon measure {\cite[Section~7.1]{Bogachev2007}}]\label{def:radon}
Let $X$ be a Hausdorff topological space. A measure $\mu$ on $X$ is called a Radon
measure if $\mu(K)<\infty$ for every compact set $K\subset X$ and, for every Borel set
$E\subset X$,
\[
\mu(E)=\sup\{\mu(F):F\subset E,\ F\text{ compact}\}.
\]
\end{definition}

\begin{definition}\label{def:sigma}
Let $X$ be a topological vector space and let $\Psi:\mathbb R\to\mathbb R$ be a
continuous squashing function. The class $\Sigma_X(\Psi)$ consists of all functions
$F:X\to\mathbb R$ of the form
\[
F(x)=\sum_{j=1}^N \omega_j\Psi(f_j(x)+b_j), \qquad x\in X,
\]
where $N\in\mathbb N$, $\omega_j,b_j\in\mathbb R$, and $f_j\in X^*$.
\end{definition}

\begin{definition}\label{def:algebra}
Let $X$ be a topological vector space and let $K\subset X$ be compact. We denote by
$\mathcal A(X^*)$ the algebra generated by the restrictions of continuous linear
functionals to $K$, that is, the smallest subalgebra of $C(K)$ containing the constant
functions and the set
\[
\{f|_K:f\in X^*\}.
\]
\end{definition}

\begin{lemma}[{\cite[Lemma~3.2]{IzukiSawano2025}}]\label{lem:hsw}
Let $\Psi$ be a squashing function. If a finite signed Borel measure $\nu$ on
$\mathbb R^n$ satisfies
\[
\int_{\mathbb R^n}\Psi(w\cdot x+b)\,d\nu(x)=0
\]
for all $w\in\mathbb R^n$ and $b\in\mathbb R$, then $\nu=0$.
\end{lemma}

\begin{lemma}[Pushforward measure {\cite[Theorem~3.6.1]{Bogachev2007}}]\label{lem:pushforward}
Let $K$ be a compact Hausdorff space, let $\mu$ be a finite signed Borel measure on $K$,
and let $F:K\to\mathbb R^m$ be continuous. Define the pushforward measure
$\nu=\mu\circ F^{-1}$ by
\[
\nu(A)=\mu(F^{-1}(A)),\qquad A\in\mathcal B(\mathbb R^m).
\]
Then $\nu$ is a finite signed Borel measure on $\mathbb R^m$, and for every bounded Borel
measurable function $G:\mathbb R^m\to\mathbb R$,
\[
\int_{\mathbb R^m}G(t)\,d\nu(t)=\int_K G(F(x))\,d\mu(x).
\]
\end{lemma}

\begin{proof}
Since $F$ is continuous, $F^{-1}(A)$ is Borel in $K$ for every Borel set
$A\subset\mathbb R^m$. Thus $\nu$ is a Borel measure. Finiteness follows from
$|\nu|(\mathbb R^m)\le |\mu|(K)<\infty$. The change-of-variables formula follows directly
from the definition of the pushforward measure, first for characteristic functions, then
for simple functions, and finally for bounded Borel functions by standard approximation.
\end{proof}

\begin{remark}\label{rem:separating-dual}
The assumption that $X^*$ separates points is automatic for normed spaces and, more
generally, for Hausdorff locally convex spaces. It is not automatic for an arbitrary
topological vector space. Thus the main theorem is naturally formulated for those
topological vector spaces whose continuous dual is sufficiently rich.
\end{remark}

\section{Main results}
We now prove the main density results. The first result establishes uniform density on compact subsets of a topological vector space whose continuous dual separates points. The proof is based on annihilating measures, finite-dimensional pushforward measures, and the Stone--Weierstrass theorem. The corresponding \(L^p\)-density result is then obtained by combining this uniform approximation theorem with the density of \(C(K)\) in \(L^p(K,\mu)\).

\begin{thm}[Density in $C(K)$]\label{thm:density-CK}
Let $X$ be a topological vector space such that $X^*$ separates points of $X$, let
$K\subset X$ be compact, and let $\Psi:\mathbb R\to\mathbb R$ be a continuous squashing
function. Then $\Sigma_X(\Psi)$ is dense in $C(K)$ with respect to the uniform norm.
\end{thm}

\begin{proof}
We argue by contradiction. Suppose that $\Sigma_X(\Psi)$ is not dense in $C(K)$.
Then its closure is a proper closed linear subspace of $C(K)$. By the Hahn--Banach
theorem, there exists a nonzero continuous linear functional $I\in C(K)^*$ such that
\[
I(h)=0 \qquad \text{for all } h\in \Sigma_X(\Psi).
\]
By Theorem \ref{thm:rmk}, there exists a nonzero finite
signed regular Borel measure $\mu$ on $K$ such that
\[
I(g)=\int_K g(x)\,d\mu(x), \qquad g\in C(K).
\]
Consequently,
\[
\int_K \Psi(f(x)+b)\,d\mu(x)=0
\]
for every $f\in X^*$ and every $b\in\mathbb R$, because each function
$x\mapsto \Psi(f(x)+b)$ belongs to $\Sigma_X(\Psi)$. Here we use the continuity of
$\Psi$ to ensure that these functions are elements of $C(K)$.

We now pass from one linear functional to finitely many linear functionals. Let
$f_1,\ldots,f_m\in X^*$ and define the continuous map
\[
F:K\to\mathbb R^m,\qquad
F(x)=(f_1(x),\ldots,f_m(x)).
\]
Let $\nu=\mu\circ F^{-1}$ be the pushforward measure on $\mathbb R^m$. Since $K$ is
compact and $F$ is continuous, the measure $\nu$ is a finite signed Borel measure with
compact support contained in $F(K)$.

For every $w=(w_1,\ldots,w_m)\in\mathbb R^m$ and every $b\in\mathbb R$, we have
\[
w\cdot F(x)+b
=
\sum_{j=1}^{m}w_j f_j(x)+b.
\]
Since $\sum_{j=1}^{m}w_j f_j\in X^*$, the annihilation property above gives
\[
\int_{\mathbb R^m}\Psi(w\cdot t+b)\,d\nu(t)
=
\int_K \Psi(w\cdot F(x)+b)\,d\mu(x)
=0.
\]
By Lemma~\ref{lem:hsw}, applied in $\mathbb R^m$, it follows that $\nu=0$.

Therefore, for every polynomial $P$ on $\mathbb R^m$,
\[
\int_K P(f_1(x),\ldots,f_m(x))\,d\mu(x)
=
\int_{\mathbb R^m}P(t)\,d\nu(t)
=0.
\]
This shows that $\mu$ annihilates every finite product and every finite linear combination
of restrictions of elements of $X^*$ to $K$. Equivalently, $\mu$ annihilates the algebra
$\mathcal A(X^*)$ generated by these restrictions.

Since $X^*$ separates points of $X$, it also separates points of $K$. Hence the algebra
$\mathcal A(X^*)$ separates points of $K$ and contains the constant functions. By Theorem \ref{thm:stone-weierstrass},
\[
\overline{\mathcal A(X^*)}^{\|\cdot\|_\infty}=C(K).
\]
On the other hand, the map
\[
g\mapsto \int_K g\,d\mu
\]
is a continuous linear functional on $C(K)$, because $\mu$ is finite. Since this functional
vanishes on the dense subalgebra $\mathcal A(X^*)$, it must vanish on all of $C(K)$.
Thus
\[
\int_K g\,d\mu=0
\qquad
\text{for every } g\in C(K).
\]
By the uniqueness part of Theorem \ref{thm:rmk}, we obtain
$\mu=0$. This contradicts the fact that $I$ was nonzero. Hence $\Sigma_X(\Psi)$ is dense
in $C(K)$.
\end{proof}

\begin{corollary}[Banach-space case]\label{cor:banach-CK}
Let $X$ be a Banach space, let $K\subset X$ be compact, and let
$\Psi:\mathbb R\to\mathbb R$ be a continuous squashing function. Then
$\Sigma_X(\Psi)$ is dense in $C(K)$ with respect to the uniform norm.
\end{corollary}

\begin{proof}
Since $X$ is a Banach space, the continuous dual $X^*$ separates points of $X$
by the Hahn--Banach theorem; see, for example, \cite{Rudin1991}. Indeed, if
$x,y\in X$ and $x\neq y$, then $x-y\neq 0$. By the Hahn--Banach theorem,
there exists $f\in X^*$ such that
\[
f(x-y)\neq 0.
\]
Hence $f(x)\neq f(y)$. Therefore the assumptions of Theorem~\ref{thm:density-CK}
are satisfied, and the conclusion follows.
\end{proof}

\begin{lemma}[Density of $C(K)$ in $L^p(K,\mu)$]\label{lem:density-C-Lp}
Let $X$ be a Hausdorff topological space, let $K\subset X$ be compact, and let $\mu$ be a
finite Radon measure on $K$. Then, for $1\le p<\infty$, the space $C(K)$ is dense in
$L^p(K,\mu)$.
\end{lemma}

\begin{proof}
Let $f\in L^p(K,\mu)$ and $\varepsilon>0$. Define
\[
f_N(x)=
\begin{cases}
f(x), & |f(x)|\le N,\\
N\operatorname{sgn}(f(x)), & |f(x)|>N.
\end{cases}
\]
Then $|f_N|\le |f|$ and $f_N\to f$ pointwise. By the dominated convergence theorem,
$\|f-f_N\|_p\to0$. Hence, for sufficiently large $N$,
\[
\|f-f_N\|_p<\frac{\varepsilon}{2}.
\]

The function $f_N$ is bounded and measurable. By Theorem \ref{thm:lusin}, for a suitable compact
set $F\subset K$, the restriction $f_N|_F$ is continuous and $\mu(K\setminus F)$ is small.
Since $F$ is closed in the compact Hausdorff space $K$, the space $K$ is normal. By Theorem \ref{thm:tietze}, there exists $g\in C(K)$ such that $g=f_N$ on $F$ and
$\|g\|_\infty\le\|f_N\|_\infty$.

Choosing $F$ so that
\[
(2\|f_N\|_\infty)^p\mu(K\setminus F)<\left(\frac{\varepsilon}{2}\right)^p,
\]
we obtain
\[
\|f_N-g\|_p^p
=
\int_{K\setminus F}|f_N(x)-g(x)|^p\,d\mu(x)
\le (2\|f_N\|_\infty)^p\mu(K\setminus F)
<\left(\frac{\varepsilon}{2}\right)^p.
\]
Therefore,
\[
\|f-g\|_p\le\|f-f_N\|_p+\|f_N-g\|_p<\varepsilon.
\]
This proves the density of $C(K)$ in $L^p(K,\mu)$.
\end{proof}

\begin{thm}[Density in $L^p(K,\mu)$]\label{thm:density-Lp}
Let $X$ be a topological vector space such that $X^*$ separates points of $X$, let
$\mu$ be a Radon probability measure with compact support $K\subset X$, let
$1\le p<\infty$, and let $\Psi:\mathbb R\to\mathbb R$ be a continuous squashing function.
Then $\Sigma_X(\Psi)$ is dense in $L^p(K,\mu)$.
\end{thm}

\begin{proof}
By Theorem~\ref{thm:density-CK}, $\Sigma_X(\Psi)$ is dense in $C(K)$ with respect to the
uniform norm. By Lemma~\ref{lem:density-C-Lp}, $C(K)$ is dense in $L^p(K,\mu)$. Since
\[
\|u\|_{L^p(K,\mu)}\le \mu(K)^{1/p}\|u\|_\infty \qquad (u\in C(K)),
\]
uniform density in $C(K)$ implies $L^p$-density in $C(K)$. Combining these two density
statements gives
\[
\overline{\Sigma_X(\Psi)}^{\|\cdot\|_{L^p(K,\mu)}}=L^p(K,\mu).
\]

\end{proof}

\begin{corollary}[Banach-space $L^p$-density]\label{cor:banach-Lp}
Let $X$ be a Banach space, let $\mu$ be a Radon probability measure with compact support
$K\subset X$, and let $1\le p<\infty$. If $\Psi:\mathbb R\to\mathbb R$ is a continuous
squashing function, then $\Sigma_X(\Psi)$ is dense in $L^p(K,\mu)$.
\end{corollary}

\begin{proof}
By the Hahn--Banach theorem, the continuous dual $X^*$ separates points of $X$. Hence the conclusion follows directly from
Theorem~\ref{thm:density-Lp}.
\end{proof}

\section{Examples and remarks}

We finish by recording several classes of spaces to which the preceding theorems apply. These examples show that the separating-dual assumption is natural in many standard settings, including Euclidean spaces, Banach spaces, and Hausdorff locally convex spaces.

\begin{example}[The Euclidean case]\label{ex:euclidean}
Let $X=\mathbb R^n$ with its usual topology, and let $K\subset\mathbb R^n$ be compact.
Then
\[
X^*=(\mathbb R^n)^*
\]
separates points of $\mathbb R^n$. Indeed, if $x,y\in\mathbb R^n$ and $x\neq y$, then
there exists $a\in\mathbb R^n$ such that
\[
a\cdot x\neq a\cdot y.
\]
Therefore, for every continuous squashing function $\Psi$, the class
\[
\Sigma_{\mathbb R^n}(\Psi)
=
\left\{
\sum_{j=1}^N \omega_j\Psi(a_j\cdot x+b_j):
N\in\mathbb N,\ \omega_j,b_j\in\mathbb R,\ a_j\in\mathbb R^n
\right\}
\]
is dense in $C(K)$. Consequently, if $\mu$ is a Radon probability measure supported on
$K$, then $\Sigma_{\mathbb R^n}(\Psi)$ is dense in $L^p(K,\mu)$ for every
$1\le p<\infty$.

This recovers the usual ridge-function form of the universal approximation theorem on
compact subsets of $\mathbb R^n$; see, for example,
\cite{Cybenko1989,Hornik1989,Pinkus1999}.
\end{example}

\begin{example}[Banach spaces]\label{ex:banach}
Let $X$ be a Banach space. Then $X^*$ separates points of $X$ by the Hahn--Banach
theorem. Hence, for every compact set $K\subset X$ and every continuous squashing
function $\Psi$, the class $\Sigma_X(\Psi)$ is dense in $C(K)$. In particular, the result
applies to standard Banach spaces such as
\[
\ell^p,\quad L^p(\Omega)\quad (1\le p<\infty),\qquad C([0,1]).
\]
provided that the approximation is considered on compact subsets of these spaces.
\end{example}

\begin{example}[Locally convex spaces]\label{ex:locally-convex}
Let $X$ be a Hausdorff locally convex topological vector space. Then continuous linear
functionals separate points of $X$; see, for example, \cite{Rudin1991}. Indeed, local
convexity and the Hausdorff property imply that points can be separated from closed
convex sets by continuous linear functionals. Therefore the main theorem applies to compact subsets of Hausdorff locally
convex spaces.

This includes, for instance, Fr\'echet spaces and many spaces of smooth or rapidly
decreasing functions equipped with their usual locally convex topologies. Thus the result
is not restricted to normed or finite-dimensional spaces.
\end{example}

\begin{remark}
The compactness assumption on $K$ is essential for the present formulation, since the
main density result is proved in the uniform norm on $C(K)$ and uses the
Stone--Weierstrass theorem. The passage to $L^p(K,\mu)$ then follows from the density
of $C(K)$ in $L^p(K,\mu)$ for Radon measures.
\end{remark}


\begin{thebibliography}{99}

\bibitem{Bogachev2007}
V.~I. Bogachev,
\textit{Measure Theory}, Vols. I--II,
Springer, Berlin--Heidelberg, 2007.

\bibitem{CapelOcariz2020}
\'A. Capel and J. Oc\'ariz,
Approximation with neural networks in variable Lebesgue spaces,
arXiv:2007.04166, 2020.

\bibitem{Cybenko1989}
G. Cybenko,
Approximation by superpositions of a sigmoidal function,
\textit{Mathematics of Control, Signals and Systems} \textbf{2} (1989), 303--314.

\bibitem{Folland1999}
G.~B. Folland,
\textit{Real Analysis: Modern Techniques and Their Applications}, 2nd ed.,
Wiley, New York, 1999.

\bibitem{Hornik1989}
K. Hornik, M. Stinchcombe, and H. White,
Multilayer feedforward networks are universal approximators,
\textit{Neural Networks} \textbf{2} (1989), no.~5, 359--366.

\bibitem{Hornik1991}
K. Hornik,
Approximation capabilities of multilayer feedforward networks,
\textit{Neural Networks} \textbf{4} (1991), no.~2, 251--257.


\bibitem{Ismailov2024}
V.~E. Ismailov,
Universal approximation theorem for neural networks with inputs from a topological vector space,
\textit{Information Processing Letters} \textbf{193} (2026), Article 106623.


\bibitem{IzukiSawano2025}
M. Izuki, T. Noi, Y. Sawano, and H. Tanaka,
Some density theorems in neural network with variable exponent,
\textit{Mediterranean Journal of Mathematics} \textbf{22} (2025), Article 180.


\bibitem{Leshno1993}
M. Leshno, V.~Ya. Lin, A. Pinkus, and S. Schocken,
Multilayer feedforward networks with a nonpolynomial activation function can approximate any function,
\textit{Neural Networks} \textbf{6} (1993), no.~6, 861--867.

\bibitem{ParkSandberg1993}
J. Park and I.~W. Sandberg,
Approximation and radial-basis-function networks,
\textit{Neural Computation} \textbf{5} (1993), no.~2, 305--316.

\bibitem{Pinkus1999}
A. Pinkus,
Approximation theory of the MLP model in neural networks,
\textit{Acta Numerica} \textbf{8} (1999), 143--195.

\bibitem{Saini2026}
S. Saini,
A universal approximation theorem for neural networks with outputs in locally convex spaces,
arXiv:2603.07242, 2026.


\bibitem{Munkres2000}
J.~R. Munkres,
\textit{Topology}, 2nd ed.,
Prentice Hall, Upper Saddle River, NJ, 2000.

\bibitem{Rudin1976}
W. Rudin,
\textit{Principles of Mathematical Analysis}, 3rd ed.,
McGraw--Hill, New York, 1976.

\bibitem{Rudin1991}
W. Rudin,
\textit{Functional Analysis}, 2nd ed.,
McGraw--Hill, New York, 1991.

\end{thebibliography}
\end{document}